\documentclass[12pt]{amsart}

\usepackage{amsmath, amsfonts, mathrsfs}
\usepackage{amssymb,latexsym}
\usepackage{enumerate}
\usepackage{bbm} 

\usepackage[top=3.1cm, bottom=3.1cm, left=3.1cm, right=3.1cm, includefoot, heightrounded]{geometry}

\newtheorem{theorem}{Theorem}[section]

\newtheorem{lemma}[theorem]{Lemma}

\theoremstyle{definition}

\numberwithin{equation}{section}

\begin{document}

\title{Another proof of the analytic Hahn - Banach Theorem}

\author{ Sokol Bush Kaliaj }

\address{
Mathematics Department, 
Science Natural Faculty, 
University of Elbasan,
Elbasan, 
Albania.
}

\email{sokol\_bush@yahoo.co.uk}

\thanks{}



\begin{abstract}
In this paper we present another proof of the analytic version of the Hahn-Banach theorem in terms of convex functionals.
\end{abstract}

\maketitle

\section{Introduction and Preliminaries}

The Hahn-Banach theorem is a powerful existence theorem whose form is particularly
appropriate to applications in linear problems.
Its principal formulations are as a dominated
extension theorem (analytic form) and as a separation theorem (geometric form). 
The following theorem is the geometric version  
of the Hahn-Banach theorem, c.f. Theorem 3.1 in \cite{SCH}, p.46.
\begin{theorem}\label{T_Separation}
Let $L$ be a topological vector space, let $M$ be a linear manifold in $L$, 
and let $G$ be a non-empty convex, open subset of $L$, 
not intersecting $M$. 
There exists a closed hyperplane in $L$, containing $M$ and not intersecting $G$.
\end{theorem}
The next theorem is the analytic version, c.f. Theorem 3.2 in \cite{SCH}, p.47.
\begin{theorem}\label{T_Dominated}
Let $M$ be a subspace of a vector space $L$ over $\Phi$ , 
let $p : L \to [0,+\infty)$  be a semi-norm. 
If $g: M \to \Phi$ is a linear functional such that 
$g$ "dominated" by $p$, i.e., 
$$
(\forall x \in M)[|g(x)| \leq p(x)],
$$
then there exists a linear functional $f: L \to \Phi$ such that
\begin{itemize}
\item
$f$ extends $g$ to $L$, i.e., 
$$
f \vert_{M} = g,
$$
\item
$f$ "dominated" by $p$, i.e., 
$$
(\forall x \in L)[|f(x)| \leq p(x)].
$$
\end{itemize}
\end{theorem}
Theorem \ref{T_Separation} implies Theorem \ref{T_Dominated}, c.f. \cite{SCH}. 
Mazur \cite{MAZ} deduced the geometric form from the analytic form.  
In this paper, we
shall present another prove of the analytic form of the Hahn-Banach Theorem.

In the paper \cite{HIR}  N.~Hirano proved a generalization of the Hahn-Banach theorem.
Kakutani \cite{KUK} gave a proof of the Hahn-Banach extension theorem by
using the Markov-Kakutani fixed-point theorem. 
There are also interesting  generalizations of the Hahn-Banach theorem in the papers 
\cite{CAK}, \cite{PLE}, \cite{RAN} and \cite{SIM}.

Throughout this paper, $L$ denotes a vector space over $\Phi$. 
The scalar field $\Phi$ is the real field $\mathbb{R}$ or the complex field $\mathbb{C}$. 
A function $p : L \to [0,+\infty)$  is said to be a \textit{semi-norm} if  
for every $x, y \in L$ and $\lambda \in \Phi$, we have
\begin{itemize}
\item
$p(x + y ) \leq p(x) + p(y)$,
\item
$p(\lambda x) = |\lambda| p(x)$. 
\end{itemize}
A function $\varphi : L \to \mathbb{R}$ is said to be a 
\textit{convex functional} if 
$0< \lambda < 1$ implies 
$$
\varphi(\lambda x + (1-\lambda) y ) \leq 
\lambda \varphi(x) + (1-\lambda) \varphi(y)
$$ 
for all $(x,y) \in L \times L$. 
Clearly, if $p$ is a semi-norm, then it is the convex functional.

\section{The Main Result}

We present another proof of the analytic form of the Hahn-Banach theorem, 
Theorem \ref{T_Dominated1}.
Let us start with the following auxiliary lemma.

\begin{lemma}\label{L_Dominated1.1} 
Let $L$ be a topological vector space over $\mathbb{R}$, 
let $\mathfrak{B}$ be the family of all circled $0$-neighborhood and  
let $\varphi : L \to \mathbb{R}$  be a convex functional. 
Then the following statements are equivalent:
\begin{itemize}
\item[(i)]
$\varphi$ is continuous on $L$,
\item[(ii)]
$\varphi$ is upper semi-continuous on $L$.
\end{itemize}
\end{lemma}
\begin{proof}
Clearly, if $\varphi$ is continuous on $L$, then $\varphi$ is upper semi-continuous on $L$. 
Conversely, assume that  $\varphi$ is upper semi-continuous at $x_{0} \in L$. 
Then, given $\varepsilon >0$ there exists $U_{\varepsilon} \in \mathfrak{B}$ such that 
\begin{equation}\label{eq_LDD.1}
u \in U_{\varepsilon}  
\Rightarrow 
\max \{\varphi(x_{0} + u), \varphi(x_{0} - u)\} < \varphi(x_{0}) + \varepsilon
\end{equation}
and since 
\begin{equation*}
2\varphi(x_{0}) \leq \varphi(x_{0} + u) + \varphi(x_{0} - u)
\end{equation*}
it follows that
\begin{equation*}
2\varphi(x_{0}) \leq \varphi(x_{0} + u) + \varphi(x_{0} - u) 
< 
2\varphi(x_{0}) + 2\varepsilon. 
\end{equation*}
The last result together with \eqref{eq_LDD.1} yields 
\begin{equation*}
u \in U_{\varepsilon}
\Rightarrow 
\varphi(x_{0}) - \varepsilon < \varphi(x_{0} + u) < \varphi(x_{0}) + \varepsilon.
\end{equation*}
This means that $\varphi$ is continuous at $x_{0}$ 
and the proof is finished.
\end{proof}

\begin{lemma}\label{L_Dominated1.2} 
Let $M$ be a subspace of a vector space $L$ over $\mathbb{R}$, 
let $\varphi : L \to \mathbb{R}$  be a convex functional such that  
for every $x \in L$ and $\lambda \geq 0$, we have
\begin{equation*}
\varphi(\lambda x) = \lambda \varphi(x). 
\end{equation*}
If $g: M \to \mathbb{R}$ is a linear functional such that 
$g$ "dominated" by $\varphi$, i.e., 
$$
(\forall x \in M)[g(x) \leq \varphi(x)],
$$
then there exists a linear functional $f: L \to \mathbb{R}$ such that
\begin{itemize}
\item[(i)]
$f$ extends $g$ to $L$, i.e., 
$$
f \vert_{M} = g,
$$
\item[(ii)]
$f$ "dominated" by $\varphi$, i.e., 
$$
(\forall x \in L)[f(x) \leq \varphi(x)].
$$
\end{itemize}
\end{lemma}
\begin{proof}
Define
$$
U_{n} = \left \{ x \in L : \varphi(x) < \frac{1}{n} \right \}, n \in \mathbb{N}
$$
and
\begin{equation*}
\mathfrak{B} = 
\left \{ V_{n} \in 2^{L} : V_{n} = U_{n} \cap (-U_{n}), n \in \mathbb{N} \right \}. 
\end{equation*}
Note that  $\mathfrak{B}$ is a filter base in $L$ such that 
\begin{itemize}
\item[(a)]
$V_{2n} + V_{2n} \subset V_{n}$, $n \in \mathbb{N}$,
\item[(b)]
every $V_{n}$ is radial and circled. 
\end{itemize} 
By Corollary of Statement 2.1 in \cite{SCH}, p.15, 
the family 
\begin{equation*}
\mathfrak{B} = 
\left \{ V_{n} \in 2^{L} : V_{n} = U_{n} \cap (-U_{n}), n \in \mathbb{N} \right \}
\end{equation*}
is a $0$-neighborhood base for a unique topology $\mathfrak{T}$ 
under which $L$ is a topological vector space. 
Fix an arbitrary vector $w_{0} \in L$. 
Given $\varepsilon >0$ there exists $n_{\varepsilon} \in \mathbb{N}$ such that 
\begin{equation*}
\begin{split}
u \in V_{n_{\varepsilon}} 
\Rightarrow 
\varphi(w_{0} + u) \leq \varphi(w_{0}) + \varphi(u) < \varphi(w_{0}) + \varepsilon.
\end{split}
\end{equation*}
Thus, $\varphi$ is upper semi-continuous at $w_{0}$, 
and since $w_{0}$ is arbitrary it follows that $\varphi$ is upper semi-continuous on $L$. 
Therefore, by Lemma \ref{L_Dominated1.1}, $\varphi$  is continuous on $L$.

Note that if $T : L \times \mathbb{R} \to \mathbb{R}$ is a linear functional, then 
\begin{equation*}
T(x,t) = h(x) + \alpha t, \text{ for all } (x,t) \in L \times \mathbb{R}, 
\end{equation*}
where 
$$
h(x) = T(x, 0),\text{ for all } x \in L  
$$
and $\alpha = T(0, 1)$. 
Hence,
\begin{equation*}
H = \{ (x,t) \in M \times \mathbb{R} : g(x) - t = 1 \}
\end{equation*}
is a linear manifold in the subspace $M \times \mathbb{R}$.  
Since $(x,t) \to g(x) - t$ is a non-zero linear functional,  
$H$ is a hyperplane in $M \times \mathbb{R}$ 
and a linear manifold in $L \times \mathbb{R}$

Since $\varphi$ is a continuous function on $L$ it follows that 
$$
G = \{ (x,t) \in L \times \mathbb{R} : \varphi(x) - t < 1 \}
$$
is an open set and 
$$
G \cap H = \emptyset, 
$$
since $g(x) \leq \varphi(x)$ for $x \in M$. 
It is easy to see that $G$ is also a non-empty convex set.
Therefore, by Theorem \ref{T_Separation} there exists a hyperplane 
$H_{1}$ in $L \times \mathbb{R}$ such that 
$$
H_{1} \supset H
\quad
\text{and}
\quad
H_{1} \cap G = \emptyset.
$$
If we choose a vector $(x_{0}, t_{0}) \in H_{1} \cap M \times \mathbb{R}$, then 
$$
H_{1} \cap (M \times \mathbb{R}) = (x_{0}, t_{0}) + H_{0} \cap (M \times \mathbb{R}), 
$$
where $H_{0}$ is a maximal subspace of $L$ such that 
$$
H_{1} = (x_{0}, t_{0}) + H_{0}.
$$
We have also 
$$
H_{1} \cap (M \times \mathbb{R}) \neq (M \times \mathbb{R}),
$$ 
since $(0,0) \not\in H_{1}$ and $(0,0) \in (M \times \mathbb{R})$.  
Thus, $H_{1} \cap (M \times \mathbb{R})$ is a hyperplane in $M \times \mathbb{R}$, 
and since $H_{1} \cap (M \times \mathbb{R}) \supset H$ 
it follows that
$$
H_{1} \cap (M \times \mathbb{R}) = H.
$$
By Statement 4.1 in \cite{SCH}, p.24, 
we have 
$$
H_{1} = \{ (x,t) \in L \times \mathbb{R}: F(x,t) = 1 \}
$$
for some non-zero linear functional $F:L \times \mathbb{R} \to \mathbb{R}$. 
Now $H_{1} \cap (M \times \mathbb{R}) = H$ implies that 
$$
F(x,0) = g(x), \text{ for all }x \in L;
$$
that is, the linear functional $f(\cdot) = F(\cdot,0)$ is an extension of $g$ to $L$. 
Since $(0,-1) \in H$ it follows that
$$
F(x,t) = f(x) -t,\text{ for }(x,t) \in L \times \mathbb{R}. 
$$

It remains to prove that $f$ "dominated" by $\varphi$. 
It is easy to see that 
\begin{equation}\label{eq_DD1}
f(x)= 1+t \leq \varphi(x), \text{ for all } (x,t) \in H_{1}. 
\end{equation}
While $(x,t) \not\in H_{1}$ implies that there exists a real number 
$r \in \mathbb{R}$ such that $r \neq 1$ and
$$
f(x) - t = r.
$$
Then, $(x, t + r -1 ) \in H_{1}$ and therefore \eqref{eq_DD1} yields
$$
f(x) \leq \varphi(x).
$$
Thus, $f$ "dominated" by $\varphi$ and this ends the proof.
\end{proof}

By virtue of Lemma \ref{L_Dominated1.2}, 
we are now able to present another proof of analytic form of the Hahn-Banach Theorem.

\begin{theorem}\label{T_Dominated1}
Let $M$ be a subspace of a vector space $L$ over $\mathbb{C}$ and 
let $p : L \to [0,+\infty)$  be a semi-norm. 
If $g: M \to \mathbb{C}$ is a linear functional such that 
$g$ "dominated" by $p$, i.e., 
$$
(\forall x \in M)[|g(x)| \leq p(x)],
$$
then there exists a linear functional $f: L \to \mathbb{C}$ such that
\begin{itemize}
\item
$f$ extends $g$ to $L$, i.e., 
$$
f \vert_{M} = g,
$$
\item
$f$ "dominated" by $p$, i.e., 
$$
(\forall x \in L)[|f(x)| \leq p(x)].
$$
\end{itemize}
\end{theorem}
\begin{proof}
There exists a linear functional $g_{0} : M \to \mathbb{R}$ such that
\begin{equation*}
g(x) = g_{0}(x) - i g_{0}(ix), \text{ for all } x \in M.
\end{equation*}
We have also that $p$ is a convex functional such that $p(\lambda x) = \lambda p(x)$ 
for every $\lambda \geq 0$ and $x \in L$. 
Thus, regarding $L$ as a real linear space, and applying Lemma \ref{L_Dominated1.1}, 
a real linear function $f_{0} : L \to \mathbb{R}$ is obtained for which 
\begin{equation*}
f_{0}\vert_{M} = g_{0}; \quad 
f_{0}(x) \leq p(x), \text{ for all } x \in L.
\end{equation*}
Let the function $f : L \to \mathbb{C}$ be defined by the 
equation 
\begin{equation*}
f(x) = f_{0}(x) - i f_{0}(i x).
\end{equation*}
Clearly, $f$ is a linear functional. 
For $x$ in $M$, we have
$$
f(x) = f_{0}(x) - i f_{0}(i x) = g_{0}(x) - i g_{0}(i x)= g(x).
$$ 
Thus $f$ is an extension of $g$. 
Finally, let $f(x) = |f(x)| e^{i \theta(x)}$, 
where $\theta(x)$ is an argument of the complex number $f(x)$; 
then
$$
|f(x)| = f_{0} \left ( \frac{x}{e^{i \theta(x)}} \right ) 
\leq 
p \left ( \frac{x}{e^{i \theta(x)}} \right ) = p(x),
$$ 
which proves that $|f( \cdot )| \leq p( \cdot )$ 
and this ends the proof.
\end{proof}

\bibliographystyle{plain}

\end{document}